\documentclass[11pt]{amsart}
\usepackage{amssymb, amsmath}
\usepackage{amsthm}
\usepackage{amscd}
\usepackage{graphicx}
\usepackage[expansion=false]{microtype}
\newtheorem{theorem}{Theorem}
\newtheorem{lemma}[theorem]{Lemma}

\newtheorem{question}[theorem]{Question}
\newtheorem{corollary}[theorem]{Corollary}
\newtheorem{problem}[theorem]{Problem}
\newtheorem{proposition}[theorem]{Proposition}

\theoremstyle{definition}

\newtheorem*{example}{Example}
\newtheorem*{remark}{Remark}

\title[Incompressible fillings of manifolds]
{Incompressible fillings of manifolds}
\author{Grigori Avramidi}
\address{Mathematische Institut\\Universit\"at M\"unster\\ Germany}
\email{avramidi@uni-muenster.de}

\def\ra{\rightarrow}

\def\beqa{\begin{eqnarray}}
\def\eeqa{\end{eqnarray}}
\def\beqa{\begin{eqnarray}}
\def\eeqa{\end{eqnarray}}

\DeclareMathOperator{\Stab}{St}
\DeclareMathOperator{\Isom}{Isom}
\DeclareMathOperator{\Out}{Out}

\DeclareMathOperator{\SL}{SL}
\DeclareMathOperator{\GL}{GL}
\DeclareMathOperator{\SO}{SO}

\begin{document}

\begin{abstract} 
We find boundaries of Borel-Serre compactifications of locally symmetric spaces, for which any filling is incompressible. We prove this result by showing that these boundaries have small singular models and using these models to obstruct compressions. We also show that small singular models of boundaries obstruct $S^1$-actions (and more generally homotopically trivial $\mathbb Z/p$-actions) on interiors of aspherical fillings. We use this to bound the symmetry of complete Riemannian metrics on such interiors in terms of the fundamental group. We also use small singular models to simplify the proofs of some already known theorems about moduli spaces (the minimal orbifold theorem and a topological analogue of Royden's theorem). 
\end{abstract}
\maketitle
\section*{Introduction}
Let $\partial$ be a closed, connected $(n-1)$-manifold. A {\it filling} of $\partial$ is a compact $n$-manifold $M$ with boundary $\partial$. It is called {\it compressible} if the manifold $M$ can be homotoped\footnote{We are allowed to move the boundary $\partial$ in the course of the homotopy.} to the boundary $\partial$. Otherwise, it is called {\it incompressible}. For instance, the disk $D^2$ is a compressible filling of $S^1$ while the torus with a disk removed $\mathbb T_1^2:=\mathbb T^2\setminus D^2$ is not. 
In dimension two, the closed orientable surfaces have compressible fillings (by handlebodies).
\begin{question}
Suppose $\partial$ is a boundary. Does it have a compressible filling? 
\end{question}
This paper shows for various boundaries $\partial$ that all $\pi_1$-injective\footnote{A filling is $\pi_1$-injective if the map $\partial\hookrightarrow M$ is a $\pi_1$-injection.} fillings $(M,\partial)$ are incompressible. In some instances we show that a boundary has no compressible fillings at all. 

\begin{theorem}
\label{incompressible}
Let $\Gamma<\SL(\mathbb Z^{4k})$ be a finite index, torsion free subgroup\footnote{For instance, the finite index subgroup $\Gamma:=\ker(\SL(\mathbb{Z}^{4k}))\ra\SL(\mathbb{F}^{4k}_3))$ is torsion free.}, and let $\partial$ be the boundary of the Borel-Serre compactification of the locally symmetric manifold $\Gamma\setminus\SL(\mathbb R^{4k})/\SO(4k)$. Then every filling of $\partial$ is incompressible.
\end{theorem} 

There are obstructions of different types. One has to do with the homotopy type of the universal cover $\widetilde\partial$. It can happen that not enough of the homology in $\widetilde\partial$ is represented by spheres to allow for a compression. This happens, for instance, for the complex projective $3$-space $\mathbb CP^3$. The other obstruction has to do with the action of the fundamental group on the universal cover $\widetilde\partial$. One can use it to show that all $\pi_1$-injective fillings of $\partial((\mathbb T^2_1)^d)$ are incompressible. It is also the main ingredient in the proof of Theorem \ref{incompressible}. 

\subsection*{$\pi_1$-injective fillings} The main step in the proof of Theorem \ref{incompressible} is showing that the $\pi_1$-injective fillings are incompressible. This is done by showing that covers $\widetilde\partial$ of the boundaries $\partial$ have singular models (spherical Tits buildings) that are small in a sense that we will now describe.  
\subsection*{Small singular models}
Let $\widetilde\partial\ra\partial$ be a connected regular cover with covering group $\Gamma$.
A {\it singular model of $\widetilde\partial$} is a $\Gamma$-complex $C$ whose Borel quotient\footnote{We use the notation $X\times_{\Gamma}Y:=(X\times Y)/\Gamma$.} $C\times_{\Gamma}E\Gamma$ is homotopy equivalent to $\partial$.\footnote{Equivalently, there is a homotopy equivalence $h:\widetilde\partial\ra C$ that is also a $\Gamma$-map.} It is {\it small} if, in addition, $C$ is a flag complex, and for any pair of disjoint simplices $\sigma$ and $\tau$ in $C$ the homological dimension stabilizers is bounded by\footnote{Inequality (\ref{one}) is a special case of (\ref{stab}) with the convention $|\emptyset|=-1$. All the arguments also work without it. However, in practice it is often helpful to have it when checking that a space is a small singular model via a vertex trading argument.} 
\begin{eqnarray}
\label{one}
\mbox{hdim}(\Stab_{\Gamma}(\sigma))+|\sigma|&\leq&\dim(\partial),\\
\label{stab}
\mbox{hdim}(\Stab_{\Gamma}(\sigma)\cap \Stab_{\Gamma}(\tau))+|\sigma|+|\tau|&<&\dim(\partial).
\end{eqnarray}
\begin{remark}
It follows from the spectral sequence 
$$
H_*(B\Gamma;C_*(C))\implies H_*(C\times_{\Gamma}E\Gamma)=H_*(\partial)
$$ 
that there are always simplices $\sigma$ in $C$ for which (\ref{one}) is an equality.
\end{remark}
\begin{theorem}
\label{nonperipheral}
If some connected $\Gamma$-cover $\widetilde\partial\ra\partial$ has a small singular model, the any $\pi_1$-injective filling $\partial\hookrightarrow M$ is incompressible.
\end{theorem}
\begin{remark}
If $\Gamma=\pi_1\partial$ and $\widetilde\partial\sim\vee S^{q-1}$ then any $\pi_1$-isomorphic filling $(M,\partial)$ has a nontrivial $(n-q)$-dimensional {\it fundamental homology class} $e$.\footnote{More precisely, it follows from the spectral sequence for the bundle $\widetilde\partial\ra\widetilde\partial\times_{\Gamma}\widetilde M\ra M$ that the group $H_{n-q}(M;H_{q-1}(\widetilde\partial))$ is infinite cyclic ($M$ does not need to be aspherical). The class $e$ is a generator of this group.} In this case, the proof of Theorem \ref{nonperipheral} can be interpreted as showing that this fundametal class $e$ is ``stuck'' in the interior of the filling and cannot be compressed to the boundary $\partial$.  
\end{remark}
Simply connected manifolds do not have small singular models. 
At the other extreme, a single point is a small singular model for the universal cover of any aspherical manifold. 
More interestingly, (some) Borel-Serre boundaries of locally symmetric manifolds and Harvey-Ivanov boundaries of moduli spaces are a source of small singular models. 
\subsection*{$(d\geq 2)$-fold Cartesian products $M:=(\mathbb T_1^2)^d=\mathbb T_1^2\times\dots\times\mathbb T_1^2$} This space has fundamental group $\Gamma=F_2^d$. Let $\partial$ be the boundary of $M$. The $\Gamma$-cover of the boundary $\widetilde\partial\ra\partial$ has a small singular model that is a $d$-fold join of infinite discrete sets. Explicitly, let $F_2^d=\left<a_1,b_1\right>\times\dots\times\left<a_d,b_d\right>$. Let $X_i:=\left<a_i,b_i\right>/\left<[a_i,b_i]\right>$ be the discrete, transitive left $\left<a_i,b_i\right>$-space\footnote{Note that this is {\it not} the abelianization $\left<a_i,b_i\right>/\left<\left<[a_i,b_i]\right>\right>$.} whose stabilizers are conjugates of the cyclic group generated by the commutator $[a_i,b_i]=a_ib_ia_i^{-1}b_i^{-1}$. It is well known that the $d$-fold join of these spaces
$$
X_1*\dots*X_d
$$
is a singular model for $\widetilde\partial$ and not hard to check that it is small.
\subsection*{Space of flat $m$-tori of unit volume}
The space of flat $m$-tori of unit volume is the orbifold $\SL(\mathbb Z^m)\setminus \SL(\mathbb R^m)/\SO(m)$. Let $\Gamma<\SL(\mathbb Z^m)$ be a finite index, torsion free subgroup and denote by $\dot{M}$ the quotient $\Gamma\setminus \SL(\mathbb R^m)/\SO(m)$. It can be compactified to a compact manifold-with-boundary $(M,\partial)$ (the Borel-Serre compactification). Its $\Gamma$-cover $(\widetilde M,\widetilde\partial)$ is homotopy equivalent to $(*,\vee S^{m-2})$, and the complex of flags in $\mathbb Q^m$ is a singular model for $\widetilde\partial$. 
\begin{proposition}
\label{flagcomplexsmall}
The flag complex of $\mathbb Q^m$ is a small singular model for $\widetilde\partial$.
\end{proposition}
\subsection*{Space of closed hyperbolic surfaces of genus $g$}
Let $\Sigma_g$ be a closed orientable surface of genus $g\geq 2$, and denote by $\Gamma$ a finite index torsion free subgroup of the extended mapping class group $\pi_0Homeo(\Sigma_g)$. Let $\dot{M}$ be the finite (orbifold) cover of moduli space\footnote{In this paper, {\it moduli space} will mean the space $\mathcal M^{\pm}_g$ of hyperbolic structures on the closed, orientable, genus $g\geq 2$ surface. It is double covered by the space of oriented hyperbolic structures.} corresponding to the subgroup $\Gamma$. It is can be compactified to a compact manifold-with-boundary $(M,\partial)$ (the Harvey-Ivanov compactification). 
Its $\Gamma$-cover $(\widetilde M,\widetilde\partial)$ is homotopy equivalent to $(*,\vee S^{2g-2})$, and the curve complex of $\Sigma_g$ is a singular model for $\widetilde\partial$.  
\begin{proposition}
\label{curvecomplexsmall}
The curve complex is a small singular model for $\widetilde\partial$. 
\end{proposition} 

Suppose $m\geq 4$ and $\Gamma<\SL(\mathbb Z^m)$ is a finite index torsion free subgroup. Then Theorem \ref{nonperipheral}, Proposition \ref{flagcomplexsmall} and the Margulis normal subgroup theorem imply that the only potentially compressible fillings of the Borel-Serre boundary $\partial(\Gamma\setminus\SL(\mathbb R^m)/\SO(m))$ are simply connected (see subsection \ref{finish}).  
\subsection*{Simply connected fillings}
Simply connected fillings of non-simply connected boundaries are often incompressible, and this can be detected by comparing the effect of a hypothetical compression on homotopy groups with its effect on homology groups (see Proposition \ref{simplyconnected}). Using this method, we show that simply connected fillings of $\partial((\mathbb T^2_1)^d)$ are incompressible for $d\geq 2$, and also that simply connected fillings of Borel-Serre boundaries corresponding to finite index torsion free subgroups of $\SL(\mathbb Z^{4k})$ 
 are incompressible, establishing Theorem \ref{incompressible}.  

\begin{remark}
Pettet and Souto showed \cite{pettetsouto} that Borel-Serre compactifications of finite volume, complete, nonpositively curved locally symmetric manifolds are always incompressible fillings (of their Borel-Serre boundaries). Their result was a major motivation for this paper. Their approach is to show that maximal flat tori are ``stuck'' in the interior of a locally symmetric manifold and cannot be homotoped out to the end. This does not work for moduli space because the maximal abelian subgroups are peripheral. So one has to find something else that is ``stuck'' in the interior, namely the fundamental class $e$. An interesting feature of the resulting argument turned out to be that it only depends on the boundary $\partial$ (it has to have a small singular model) and not at all on the topology of the filling. So, even for the Borel-Serre boundaries of locally symmetric manifolds $\Gamma\setminus\SL(\mathbb R^{m})/\SO(m)$ the main result of this paper is new because it applies to any filling (not just the locally symmetric one). 
\end{remark}
 
 
\subsection*{Homotopically trivial $\mathbb Z/p$-actions}
Let $(M,\partial)$ be a compact manifold-with-boundary. If $M$ is aspherical, then the compressibility question is closely related to the following question. 
\begin{question}
Does the interior $\dot{M}$ of the manifold $M$ have any $S^1$-actions? More generally, does it have any homotopically trivial $\mathbb Z/p$-actions?
\end{question}
If $M$ is an aspherical manifold with centerless fundamental group\footnote{If the fundamental group has trivial center, any homotopically trivial action on $\dot M$ lifts to the universal cover, and one can usefully apply Smith theory to the lifted action.}, then Smith theory can be used to localize the $\mathbb F_p$-topology of the interior $\dot{M}$ to the fixed point set $F$ of a homotopically trivial $\mathbb Z/p$-action. From the point of view of $\mathbb F_p$-homology, the interior $\dot{M}$ looks like a regular neighborhood of $F$. This can be used to obstruct homotopically trivial $\mathbb Z/p$-actions when the $\pi_1M$-cover of the boundary $\widetilde\partial$ has a small singular model. We do this (using the method developed in \cite{avramidi2011periodic}) when $\widetilde\partial$ is homotopy equivalent to a (possibly infinite) wedge of spheres $\vee S^{q-1}$ of a single dimension $q-1$.\footnote{This seems to be a simplifying assumption and we expect the theorem to be true without it. The examples of small singular models in this paper have $\widetilde\partial\sim\vee S^{q-1}$. If one finds natural examples without this concentration, then it might be worthwhile to try and prove a more general version of this theorem.} 
\begin{theorem}
\label{htrivial}
Suppose the pair $(\widetilde M,\widetilde\partial)$ is homotopy equivalent to $(*,\vee S^{q-1})$ and $\widetilde\partial$ has a small singular model. Suppose further that $\pi_1M$ is centerless. Then any homotopically trivial $\mathbb Z/p$-action on the interior $\dot{M}$ of $M$ is trivial. 
\end{theorem}
\subsection*{Applications to moduli spaces}
One consequence of Theorem \ref{htrivial} (using work of Ivanov) is that moduli space is a minimal orbifold\footnote{It is not a finite orbifold cover of a smaller orbifold.}. A further consequence (using work of Farb and Weinberger) is that for any complete, finite volume, Riemannian (or Finsler) metric $h$ on moduli space, the isometry group of its lift $\widetilde h$ to the universal cover is the extended mapping class group (see section \ref{mcgcorollaries}).

\subsection*{Symmetries of $\dot M$} Let $g$ be a complete Riemannian metric on the interior $\dot{M}$. Theorems \ref{nonperipheral} and \ref{htrivial} can be used to bound the isometry group of $g$ purely in terms of the fundamental group. Isometries permute loops in $\dot{M}$ and this gives a homomorphism 
\begin{equation}
\rho:\Isom(\dot{M},g)\ra\Out(\pi_1M).
\end{equation}
\begin{corollary}
\label{isometries}
Let $g$ be a complete Riemannian metric on the interior $\dot{M}$. If $(M,\partial)$ satisfies the assumptions of Theorem \ref{htrivial}, then $\rho$ is injective. In other words, $g$ has no homotopically trivial isometries. 
\end{corollary}
 
\subsection*{The obstruction $\alpha_M$}
Homotoping $M$ into its boundary is the same thing as finding a homotopy section $M\ra\partial$ of the inclusion $\partial\hookrightarrow M$. The primary obstruction to doing this as a cohomology class $\alpha_M\in H^q(M;D)$ with coefficients in the $\Gamma$-module $D:=\overline H_{q-1}(\widetilde\partial)$. 
\begin{proposition}
\label{obstructionproposition}
Under the assumptions of Theorem \ref{htrivial}, $\alpha_M\not=0$. 
\end{proposition}

\section{A computation in $\Gamma$-equivariant homology}
\label{equivariant}
In this section, homology is with coefficients in $\mathbb Z$ or $\mathbb F_p$. Give $C\times C$ the diagonal $\Gamma$-action. 
Recall that {\it $\Gamma$-equivariant homology} $H^{\Gamma}_*$ of a $\Gamma$-complex $X$ is defined as
\begin{equation}
H^{\Gamma}_*(X):=H_*(X\times_{\Gamma} E\Gamma).
\end{equation}
The goal of this section is the following Proposition.
\begin{proposition}
\label{diagonal}
If $C$ is a flag complex and for every pair of disjoint simplices $\sigma$ and $\tau$ in $C$ we have 
\begin{equation}
hdim(\Stab_{\Gamma}(\sigma)\cap\Stab_{\Gamma}(\tau))+|\sigma|+|\tau|<n-1,
\end{equation}
then the diagonal map $C\stackrel{s}\hookrightarrow C\times C$ induces an isomorphism
\begin{equation}
\label{equiv}
H_{\geq n-1}^{\Gamma}(C)\cong H_{\geq n-1}^{\Gamma}(C\times C).
\end{equation}
\end{proposition}
\subsection*{The simplicial diagonal}
We will use the {\it simplicial diagonal} 
\begin{equation}
\Delta:=\bigcup_{\sigma}(\sigma\times\sigma).
\end{equation} 
Clearly the diagonal inclusion $s$ factors as $C\hookrightarrow\Delta\hookrightarrow C\times C$. 
Moreover, if $p_1:C\times C\ra C$ is the projection to the first factor and $p:=p_1\mid_{\Delta}$ is its restriction to the simplicial diagonal, then the fibre $p^{-1}(v)$ over a {\it vertex} of $C$ is $v\times\Delta_v$, where $\Delta_v$ is the union of all simplices in $C$ containing $v$. It can also be expressed as $\Delta_v=v*Lk(v)$. 
More generally, for any simplex $\sigma$ in $C$, the inverse image of the interior $\dot{\sigma}$ of that simplex is
\begin{equation}
p^{-1}(\dot{\sigma})=\dot{\sigma}\times\Delta_{\sigma},
\end{equation}
where $\Delta_{\sigma}$ is the union of all simplices in $C$ containing $\sigma$, and can also be expressed as
$$
\Delta_{\sigma}=\sigma*Lk(\sigma).
$$

\subsection*{Proof outline}
The proof of the Proposition consists of three steps. 
\begin{itemize}
\item
Show that the inclusion $s:C\hookrightarrow\Delta$ is an $H^{\Gamma}_*$-isomorphism,
\item
Prove the fiberwise\footnote{We call this fiberwise vanishing because $p^{-1}(x)\cap(C\times C,\Delta)=x\times(C,\Delta_{\sigma})$ where $x\in\dot{\sigma}$ is a point in the interior of $\sigma$.} vanishing result 
\begin{equation}
\label{fibervanish}
H_{\geq n-1-|\sigma|}^{\Stab_{\Gamma}(\sigma)}(C,\Delta_{\sigma}),
\end{equation}
\item
show there is a spectral sequence
\begin{equation}
\bigoplus_{|\sigma|=i}H^{\Stab_{\Gamma}(\sigma)}_{j}(C,\Delta_{\sigma})\implies H^{\Gamma}_{i+j}(C\times C,\Delta)
\end{equation}
which assembles the fiberwise vanishing (\ref{fibervanish}) into the global vanishing 
\begin{equation}
H^{\Gamma}_{\geq n-1}(C\times C,\Delta)=0.
\end{equation}
\end{itemize}
This implies that $\Delta\hookrightarrow C\times C$ is an $H^{\Gamma}_{\geq n-1}$-isomorphism, which together with the first bullet proves the Proposition. 
\subsection*{Proof of first bullet}
For a pair of points $x,y\in\Delta$, there is a unique simplex $\sigma$ so that $x\in\dot\sigma$ and $y\in\Delta_{\sigma}$. The points $y$ and $x$ are connected by a straight line segment in $\Delta_{\sigma}$. So, we can define a ``straight line'' homotopy $h_t:\Delta\ra\Delta$ with $h_0(x,y)=(x,y)$ and $h_1(x,y)=(x,x)$. This homotopy is clearly a $\Gamma$-map. Therefore, the identity map $id_{\Delta}$ is $\Gamma$-homotopic to a map that factors as a composition of $\Gamma$-maps $\Delta\stackrel{p}\ra C\stackrel{s}\hookrightarrow\Delta$, where $p=p_1\mid_{\Delta}$ is the projection to the first factor restricted to $\Delta$. On the other hand, the composition $C\stackrel{s}\hookrightarrow\Delta\stackrel{p}\ra C$ is equal to the identity map $id_C$. Therefore $s$ is a $\Gamma$-homotopy equivalence, so it induces an $H^{\Gamma}_*$-isomorphism. 
\subsection*{Proof of second bullet}
Denote by $N(\sigma)=\cup_{\sigma\cap\tau\not=\emptyset}\tau$ the union of the simplices intersecting $\sigma$. The bound on stabilizers (\ref{stab}) for disjoint simplices $\sigma$ and $\tau$ can be rewritten as
\begin{equation}
\mbox{hdim}(\Stab_{\Stab_{\Gamma}(\sigma)}(\tau))+|\tau|<n-1-|\sigma|,
\end{equation}
and this implies that for $i+j\geq n-1-|\sigma|$ we have
\begin{eqnarray*}
H_{i}(B\Stab_{\Gamma}(\sigma);C_j(C,N(\sigma)))&=&\bigoplus_{|\tau|=j}H_i(B\Stab_{\Gamma}(\sigma);\mathbb Z[\Stab_{\Gamma}(\sigma)/\Stab_{\Stab_{\Gamma}(\sigma)}(\tau)])\\
&=&\bigoplus_{|\tau|=j}H_i(B\Stab_{\Stab_{\Gamma}(\sigma)}(\tau)),\\
&=&0,
\end{eqnarray*}
where the sum is over $\Stab_{\Gamma}(\sigma)$-orbits of $j$-simplices $\tau$ that are disjoint from $\sigma$.  
Putting this into the equivariant homology spectral sequence\footnote{For a a group $\Gamma$ and $\Gamma$-complexes $Y\subset X$ this is the spectral sequence $$H_{i}(B\Gamma;C_j(X,Y))\implies H^{\Gamma}_{i+j}(X,Y)$$ of the fible bundle $(X,Y)\ra(X\times_{\Gamma}E\Gamma,Y\times_{\Gamma}E\Gamma)\ra B\Gamma$.} for the pair $(C,N(\sigma))$ gives 
\begin{equation}
\label{relative}
H^{\Stab_{\Gamma}(\sigma)}_{\geq n-1-|\sigma|}(C,N(\sigma))=0.
\end{equation}
Since $C$ is a flag complex, the union $N(\sigma)$ is acyclic. The set $\Delta_{\sigma}=\sigma*Lk(\sigma)$ is clearly a $\Stab_{\Gamma}(\sigma)$-invariant subset of $N(\sigma)$, and it is also acyclic. Therefore, the inclusion $\Delta_{\sigma}\hookrightarrow N(\sigma)$ is an $H^{\Stab_{\Gamma}(\sigma)}_*$-isomorphism. 
Together with (\ref{relative}), this implies
\begin{equation}
\label{relzeros}
H^{\Stab_{\Gamma}(\sigma)}_{\geq n-1-|\sigma|}(C,\Delta_{\sigma})=0.
\end{equation}


\subsection*{Notation}
Below, we will be looking at the {\it product chains} on $C\times C\times E\Gamma$ and various subcomplexes formed out of this. As $\Gamma$-modules, all chain complexes will be submodules of the the product $C_*(C)\otimes C_*(C)\otimes C_*(E\Gamma)$. The differentials on these complexes will always be obtained from the differentials on $C_*(C)$ and $C_*(E\Gamma)$ by forming total complexes, restricting to subcomplexes and/or taking quotients. In the proof below we will use the following convention: $C_*(X\times Y)$ will always denote the total chain complex $C_*(X\times Y)=C_*(X)\stackrel{Tot}\otimes C_*(Y)$, i.e. the complex of product chains on $X\times Y$. 

\subsection*{On computing relative equivariant homology}
Suppose that $Y$ is a $\Gamma$-subcomplex of a {\it free} $\Gamma$-complex $X$. Then the relative equivariant homology $H^{\Gamma}_*(X,Y)$ can be computed as the ordinary homology of the complex of coinvariants of the relative cellular chain complex $C_*(X,Y)_{\Gamma}$. The reason is that, when we use cellular chains, applying coinvariants to 
\begin{equation}
\label{relative}
0\ra C_*(Y)\ra C_*(X)\ra C_*(X,Y)\ra 0
\end{equation}
preserves exactness because, as a sequence of $\Gamma$-modules (not $\Gamma$-chain complexes!) the sequence (\ref{relative}) splits: the $\Gamma$-module $C_i(X)$ is a direct sum of $\Gamma$-modules $C_i(X)=C_i(Y)\oplus C_i(X,Y)$ since $C_i(X,Y)$ is generated by the $i$ cells that do not lie in $Y$. 

More generally, if $X$ is a $\Gamma$-complex that is not necessarily free, then we can compute $H^{\Gamma}_*(X,Y)$ as the homology of the complex $C_*(X\times E\Gamma,Y\times E\Gamma)_{\Gamma}$ where $C_*(X\times E\Gamma,Y\times E\Gamma)$ is the relative complex of cellular product chains.  

\subsection*{Proof of third bullet}
Let 
$$
C_*(C\times C\times E\Gamma)=C_*(C)\stackrel{Tot}\otimes (C_*(C)\stackrel{Tot}\otimes C_*(E\Gamma))
$$ 
be the complex of product chains and $C_*(\Delta\times E\Gamma)$ the subcomplex of those product chains that are supported on $\Delta\times E\Gamma$. Our goal will be to compute the $\Gamma$-equivariant homology of the pair $(C\times C,\Delta)$. 
We will do this via a double complex spectral sequence of a complex $(E_{*,*})_{\Gamma}$ described below.  

Both of the differentials of the double complex $C_*(C)\otimes C_*(C\times E\Gamma)$ preserve $C_*(\Delta\times E\Gamma)$, so it is also the total complex of a double complex which we will denote by $C_{*,*}(\Delta\times E\Gamma)$. The (single) complex $C_{i,*}(\Delta\times E\Gamma)$ decomposes as a direct sum 
\begin{equation}
\label{decomp}
C_{i,*}(\Delta\times E\Gamma)=\bigoplus_{|\sigma|=i}C_*(\Delta_{\sigma}\times E\Gamma)
\end{equation}
over $i$-simplices of $C$, because a product $\sigma\otimes\alpha\otimes\beta\in C_i(C)\otimes C_*(C\times E\Gamma)$ that lies in $C_{i,*}(\Delta\times E\Gamma)$ is precisely a product whose second factor $\alpha$ is supported on $p^{-1}(x)=\Delta_{\sigma}$, where $x\in\dot{\sigma}$ is a point in the interior of $\sigma$.
Now, take the quotient double complex 
\begin{equation}
\label{double}
E_{*,*}:={C_*(C)\otimes C_*(C\times E\Gamma)\over C_{*,*}(\Delta\times E\Gamma)}.
\end{equation}
The two bullets below follow directly from (\ref{decomp}) and the definition (\ref{double}).
\begin{itemize}
\item
$E_{i,*}$ decomposes as a direct sum of relative chain complexes
$$
E_{i,*}=\bigoplus_{|\sigma|=i}C_*(C\times E\Gamma,\Delta_{\sigma}\times E\Gamma).
$$ 
\item
The total complex of $E_{*,*}$ is $Tot(E_{*,*})=C_*(C\times C\times E\Gamma,\Delta\times E\Gamma)$. 
\end{itemize}

%


Next, we look at the $\Gamma$-action on $E_{*,*}$. 
\begin{itemize}
\item
The $\Gamma$-coinvariants of $E_{*,*}$ are the double complex with $(i,j)$-th entry
\begin{eqnarray}
(E_{i,j})_{\Gamma}&=&\bigoplus_{|\sigma|=i}C_j(C\times E\Gamma,\Delta_{\sigma}\times E\Gamma)_{\Stab_{\Gamma}(\sigma)}.
\end{eqnarray}
where the sum is now over {\it $\Gamma$-orbits} of $i$-simplices $\Gamma\cdot\sigma$ in $C/\Gamma$.
\item
The total complex of $(E_{*,*})_{\Gamma}$ is (obviously) the same as the $\Gamma$-invariants of the total complex of $E_{*,*}$, so  
\begin{eqnarray}
Tot((E_{*,*})_{\Gamma})&=&(Tot(E_{*,*}))_{\Gamma},\\
&=&C_*(C\times C\times E\Gamma,\Delta\times E\Gamma)_{\Gamma}.
\end{eqnarray}
\end{itemize}
So, the\footnote{Actually, one of the two spectral sequences associated to this double complex.} spectral sequence associated to $(E_{*,*})_{\Gamma}$ has the form
\begin{equation}
\label{doublespectralseuqnece}
(E_{i,j})^1_{\Gamma}=\bigoplus_{|\sigma|=i}H_j^{\Stab_{\Gamma}(\sigma)}(C,\Delta_{\sigma})\implies H^{\Gamma}_{i+j}(C\times C,\Delta).
\end{equation}
We conclude from the second bullet that the left hand side of the spectral sequence (\ref{doublespectralseuqnece}) vanishes for $i+j\geq n-1$ and consequently 
\begin{equation}
H^{\Gamma}_{\geq n-1}(C\times C,\Delta)=0.
\end{equation}
This finishes the proof of the Proposition.

\section{No homotopy section (Proof of Theorem \ref{nonperipheral})} 
Next, we use Proposition \ref{diagonal} to obstruct homotopy sections of $\partial\hookrightarrow M$.  

\begin{proposition}
\label{nosection}
Suppose $\widetilde\partial\ra\partial$ is a connected $\Gamma$-cover and let $\partial\hookrightarrow M$ be a $\pi_1$-injective filling. If $\widetilde\partial$ has a small singular model $C$, then $\partial\hookrightarrow M$ has no homotopy section. 
\end{proposition}

\begin{proof}
Let $C$ be a small singular model for $\widetilde\partial$. Then there is a homotopy equivalence $h:\widetilde\partial\ra C$ that is also a $\Gamma$-map. Such a map induces an $H^{\Gamma}_{*}$-isomorphism 
\begin{equation}
H_{*}(\partial)\cong H^{\Gamma}_*(\widetilde\partial)\stackrel{h_*}\cong H^{\Gamma}_*(C).
\end{equation} 
The product map $h\times h:\widetilde\partial\times\widetilde\partial\ra C\times C$ is also a homotopy equivalence and a $\Gamma$-map, so it induces an $H^{\Gamma}_{*}$-isomorphism
\begin{equation}
H_*(\widetilde\partial\times_{\Gamma}\widetilde\partial)\cong H^{\Gamma}_*(\widetilde\partial\times\widetilde\partial)\stackrel{(h\times h)_*}\cong H^{\Gamma}_*(C\times C).
\end{equation}

Now, let 
\begin{eqnarray*}
\delta:\partial&\ra&\widetilde\partial\times_{\Gamma}\widetilde\partial,\\
x&\mapsto&(x,x)
\end{eqnarray*} be the diagonal map. Look at the commutative diagram 
\begin{equation}
\begin{array}{ccccccc}
H_{n-1}^{\Gamma}(C\times C)&\cong&H_{n-1}(\widetilde\partial\times_{\Gamma}\widetilde\partial)&\stackrel{i_*}\ra&H_{n-1}(\widetilde\partial\times_{\Gamma}\widetilde M)\\
\tiny{s_*}\uparrow&&\tiny{\delta_*}\uparrow&\nearrow&\\
H_{n-1}^{\Gamma}(C)&\cong&H_{n-1}(\partial)&&\\
&&||&&\\
&&\mathbb Z&&
\end{array}
\end{equation}
The map $s_*$ is an isomorphism by Proposition \ref{diagonal}, so $\delta_*$ is also an isomorphism. The composition $i_*\delta_*$ is injective because it is induced by a section of the bundle $\widetilde M\ra\widetilde\partial\times_{\Gamma}\widetilde M
\ra\partial$. Since $\delta_*$ is an isomorphism, the map $i_*$ is also injective. 

Suppose there is a homotopy section $c:M\ra\partial$ of $\partial\hookrightarrow M$. Then $i_*$ is surjective (because it has a section of the form $(x,y)\mapsto(x,cy)$) 
 and then all homology groups in the diagram above are isomorphic to $\mathbb Z$. So, the map 
\begin{eqnarray*}
\widetilde\partial\times_{\Gamma}\widetilde\partial&\ra&\widetilde\partial\times_{\Gamma}\widetilde\partial,\\
(x,y)&\mapsto&(x,cy),
\end{eqnarray*} 
is an $H_{n-1}$-isomorphism, and hence the flipped map $(x,y)\mapsto(cx,y)$ is, as well\footnote{Because $(x,y)\mapsto(y,x)$ is an isomorphism.}. So, the composite $x\mapsto(x,x)\mapsto(x,cx)\mapsto(cx,cx)$ is an $H_{n-1}$-isomorphism. But, this composite map factors as $\partial\hookrightarrow M\ra\widetilde\partial\times_{\Gamma}\widetilde\partial$, so it is zero on $H_{n-1}$ because $\partial$ bounds $M$. This gives a contradiction, so there is no homotopy section $c$.   
\end{proof}

Proposition \ref{nosection} is a restatement of Theorem \ref{nonperipheral}, so we are now done with the proof of that theorem. 
\subsection{\label{variant}A slight variant} Note that the same argument proves the following variant of Theorem \ref{nonperipheral}: {\it If $\partial$ is connected, $(M,\partial)$ is a filling and the $\pi_1M$-cover $\widetilde\partial$ of $\partial$ has a small singular model, then $(M,\partial)$ is incompressible.} The point here is that $\partial\hookrightarrow M$ doesn't need to be $\pi_1$-injective as long as the cover $\widetilde\partial\ra\partial$ that has a small singular model comes from the filling $M$. 

\section{The curve complex is small (Proof of Proposition \ref{curvecomplexsmall})}
We need to verify inequality (\ref{stab}) for the curve complex $C$. Let $\Stab^0_{\Gamma}(\sigma)$ be the {\it pointwise} stabilizer of a simplex $\sigma$ in $C$. It is a finite index subgroup of the stabilizer $\Stab_{\Gamma}(\sigma)$. Therefore, $\Stab^0_{\Gamma}(\sigma)\cap\Stab^0_{\Gamma}(\tau)$ is a finite index subgroup of $\Stab_{\Gamma}(\sigma)\cap\Stab_{\Gamma}(\tau)$. Since the group $\Gamma$ is torsion-free, these two groups have the same homological dimension. So, we can use poinwise stabilizers when verifying (\ref{stab}). In the case of the curve complex the inequality for pointwise stabilizers we need to prove has the following form.  
\begin{lemma}
\label{smallstabilizers}
Let $A$ be a collection of disjoint simple closed curves and $B$ another collection of disjoint simple
closed curves, such that no two curves in $A\cup B$ are homotopic. Then
\begin{equation}
\label{ccstab}
\mathrm{hdim}(\mathrm{Stab}_{\Gamma}(A\cup B))+|A|-1+|B|-1<6g-7.
\end{equation}
\end{lemma}
For the proof we recall the virtual homological dimension of various mapping class groups, which was computed by Harer in \cite{harer}. Let $d(g,r,s)$ be the virtual homological dimension of a connected genus $g$ surface with $r$ punctures\footnote{For purposes of mapping class groups, a puncture is the same as a boundary component that is not fixed.} and $s$ boundary components. For $2g+s+r>2$
\begin{eqnarray}
d(g,0,0)&=&4g-5,\\
d(0,r,s)&=&(2r+s)-3,\\
d(g,r,s)&=&4g-4+(2r+s) \mbox{ if } g>0,r+s>0.
\end{eqnarray}
\begin{itemize}
\item
If $\Gamma^s_{g,r}$ is the mapping class group of a genus $g$ surface with $s$ punctures and $r$ boundary components, then the stabilizer of a curve is either isomorphic to $\Gamma^{s+1}_{g-1,r+1}$ (if the curve is not separating) or to $\Gamma^{s_1+1}_{g_1,r_1}\times\Gamma^{s_2}_{g_2,r_2+1}$ where $g=g_1+g_2,s=s_1+s_2,$ and $r=r_1+r_2$ (if the curve is separating).  
\end{itemize}
\begin{proof}
The proof consists of several steps. First, we will
show the inequality in the case $|A|=0.$
We induct on $B,$ starting with the case when $|B|=3g-3$
is as large as possible. In this case, the curves in $B$
form a pair of pants decomposition of the surface and the stabilizer
of $B$ is the free abelian group $\mathbb Z^{3g-3}$ consisting
of Dehn twists about the curves in $B$.
Thus, 
\begin{equation}
\mbox{hdim}(\mbox{Stab}\{B\})+|A|-1+|B|-1=(3g-3)+0-1+(3g-3)-1=6g-8<6g-7.
\end{equation}
Moreover, if we remove a curve from $B,$ the homological 
dimension of the stabilizer increases at most by one,
so the inequality follows.

Now, we will deal with the case $|B|=3g-3$ and any $A$.
In this situation, the set $A$ intersects at least $|A|$
curves of $B.$ (If $A$ intersects fewer than $|A|$ curves
of $B$, then we can replace the curves of $B$ by the curves
of $A$ and get more than $3g-3$ simple closed curves on the surface,
which is a contradiction.)
Thus, the stabilizer of $A\cup B$ is at most $\mathbb Z^{3g-3-|A|}$
(since the Dehn twists about curves that intersect $A$ are no longer in the
stabilizer.) Again, this gives the inequality we need.

Finally, we do the general case by inducting on $|B|$ from the top.
Note that we can assume that $A$ intersects at least $|A|$ curves of
$B$ (if not, then replacing the curves of $B$ by curves of $A$ we get 
a set $B'$ with more curves than $B$.)
Thus, again the stabilizer of $A\cup B$ has homological dimension $\leq\mbox{hdim}(\mbox{Stab}_{\Gamma}(B))-|A|$,
and the inequality follows.
\end{proof} 

\section{The complex of flags in $\mathbb Q^{m}$ is a small singular model (Proof of Proposition \ref{flagcomplexsmall})}
 
Let $C$ be simplicial complex whose $(r-1)$-simplices are flags $V_1\subset\dots\subset V_r$ in $\mathbb Q^m$. It is a spherical building modeled on $(S^{m-2},W)$, where $W=S_m$ is the symmetric group acting on the sphere $S^{m-2}=\{(x_1,\dots,x_m)\in\mathbb R^m\mid x_1^2+\dots+x_m^2=1, x_1+\dots+x_m=0\}$ by permuting coordinates. For any pair of simplices $\sigma$ and $\tau$, there is an $(m-2)$-sphere $S$ in $C$ containing both of them. The fundamental domain for the $W$-action on this sphere $S$ is a spherical $(m-2)$-simplex given by a complete flag
\begin{equation}
\label{flag}
\left<e_1\right>\subset\left<e_1,e_2\right>\subset\dots\subset\left<e_1,\dots,e_{m-1}\right>.
\end{equation} 
A flag is {\it standard} if it is a subflag of (\ref{flag}). 
Denote by $\Stab(E_*)$ the stabilizer in $\GL(\mathbb R^m)$ of the flag $E_*\otimes_{\mathbb Q}\mathbb R$. If $E_*$ is a standard flag then, in the ordered basis $\{e_1,\dots,e_m\}$ it consists of block upper triangular matrices.

\begin{lemma}
\label{upper}
Suppose $F_1\subset\dots\subset F_r$ is a standard flag, $w\in W$ is a permutation and $wF_k\not= F_k$ for every $k$. Then, there are at least $r$ matrix entries above the diagonal that are identically zero on $\Stab(wF_*)$.\footnote{In other words, there are at least $r$ pairs $i<j$ so that $g_{i,j}=0$ for all $g\in\Stab(wF_*)$.} 
\end{lemma}
\begin{proof}
Let $|V|$ be the dimension of the vector space $V$. Since $w$ does not preserve $F_k$, there is $j_k\leq|F_k|$ with $w(j_k)>|F_k|$. 
Let $I_k:=\{i>|F_k| \mbox{ s.t. } w(i)<w(j_k)\}$. Then,
\begin{itemize}
\item
the number of elements in $I_k$ is $\#I_k\geq w(j_k)-|F_k|$,\item
if $i\in I_k$ and $h\in\Stab(F_*)$ then $h_{i,j_k}=0$, so
\item
if $i\in I_k$ and $g\in\Stab(wF_*)$,
then $g_{w(i),w(j_k)}=0$.  
\end{itemize}
So for any $i\in I_k$,  $(w(i),w(j_k))$ is an above diagonal matrix entry that is identically zero on $G$. 
Also, if there is a smallest index $l$ such that $w(j_k)\leq|F_l|$, then $w(j_k)-|F_k|\geq l-k$. If there is no such index then $w(j_k)>|F_r|$ and consequently $w(j_k)-|F_k|>r-k$. 

So, let $k(1)=1$, let $k(t+1)$ be the smallest index such that $w(j_{k(t)})\leq|F_{k(t+1)}|$ (when such an index exists) and let $k(q)$ be the last index. Then
\begin{equation}
\begin{array}{ccccccccc}
r-1&=&(r-k(q))&+&(k(q)-k(q-1))&+&\dots&+&(k(2)-k(1))\\
&<&\#I_{k(q)}&+&\#I_{k(q-1)}&+&\dots&+&\#I_{k(1)}.
\end{array}
\end{equation}
Since $w(j_{k(q)})>w(j_{k(q-1)})>\dots>w(j_{k(1)})$ are all distinct, we conclude that there are at least $r$ above diagonal matrix entries that are identically zero on $\Stab(wF_*)$. 
\end{proof}
\begin{corollary}\footnote{I.e., in the space of flags in $\mathbb R^m$, the $\Stab(E_*)$-orbit of $F_*$ has dimension $\geq$ length$(F_*)$.}
\label{orbit}
If two flags $E_*$ and $F_*$ are disjoint\footnote{This means $E_i\not=F_j$ for any $i,j$.} then
\begin{equation}
\label{orbdim}
\dim(\Stab(E_*)/(\Stab(E_*)\cap\Stab(F_*)))\geq\mbox{ length}(F_*).
\end{equation}  
\end{corollary}
\begin{proof}
We may assume that $E_*$ is a standard flag, and $F_*$ is a translate of a standard flag by some element $w\in W$. Moreover, if an element $F_k$ in $F_*$ is standard, then we can move it from $F_*$ to $E_*$. This decreases the right hand side of inequality $(\ref{orbdim})$ by one and the left hand side by at least one. So, we may assume none of the elements of $F_*$ are standard, i.e. that $wF_k\not=F_k$ for all $k$. By Lemma \ref{upper}, there are at least $r=$ length$(F_*)$ matrix entries above the diagonal that are identically zero on $\Stab(F_*)$. On the other other hand, $\Stab(E_*)$ is a group of block upper triangular matrices all of whose above diagonal entries are non-zero. Consequently, $\Stab(E_*)\cap\Stab(F_*)$ is a submanifold in $\Stab(E_*)$ of codimension $\geq r$.  
\end{proof}
\subsection*{Splitting the stabilizer}
Write a flag $E_*$ as 
$$
0\subsetneq E_1\subsetneq\dots\subsetneq E_r\subsetneq\mathbb Q^m.
$$
The number of distinct subspaces $E_i$ is the length of $E_*$, denoted $length(E_*)$. (Above, $length(E_*)=r$).
We will also use the conventions $E_0=0$ and $E_{r+1}=\mathbb Q^{m}$. The associated graded vector space of the flag $E_*$ is 
$$
Gr(E_*):=\bigoplus_{i=0}^{r}E_{i+1}/E_i=E_1\oplus (E_2/E_1)\oplus\dots\oplus (E_{r}/E_{r-1})\oplus(\mathbb Q^m/E_{r}).
$$
The stabilizer $\Stab(E_*)$ acts on this graded vector space $Gr(E_*)$ and the subgroup $N$ of elements that act trivially is nilpotent. Associated to this action is a splitting of $\Stab(E_*)$ as a semidirect product 
\begin{equation}
\label{stsplitting}
\Stab(E_*)=N\rtimes\left(\prod_{i=0}^r \GL(E_{i+1}/E_i)\right).
\end{equation}


\subsection*{Intersections of flags} Now, denote by $E_*$ and $F_*$ the flags corresponing to simplices $\sigma$ and $\tau$, respectively.\footnote{Then $length(E_*)=|\sigma|+1$ and $length(F_*)=|\tau|+1$.} 
The flag $F_*$ defines a flag $F_{i*}:=(E_{i+1}\cap F_*)/(E_i\cap F_*)$ in each of the vector spaces $E_{i+1}/E_i$. 
Since $\sigma$ and $\tau$ are disjoint, the vector space $F_j$ is not part of the flag $E_*$, so it is not preserved by the stabilizer $\Stab(E_*)$. Consequently, one (or both) of the following must be true (otherwise it would follow from (\ref{stsplitting}) that every element of $\Stab(E_*)$ preserves $F_j$): 
\begin{enumerate}
\item
\label{nontrivial}
$F_{i,j}=(E_{i+1}\cap F_j)/(E_i\cap F_j)$ a non-trivial proper subspace of $E_{i+1}/E_i$ for some $i$, or
\item
the nilpotent group $N$ does not preserve $F_j$, i.e. $NF_j\not=F_j$.
\end{enumerate}
Let $F^0_*$ be the subflag of $F_*$ consisting of those spaces $F_j$ that do not satisfy (\ref{nontrivial}). By Corollary \ref{orbit}, the $\Stab(E_*)$-orbit of $F^0_*$ has dimension $\geq$ length$(F_*^0)$. Since the $M_i$ fix $F_*^0$, equation (\ref{stsplitting}) implies the $N$-orbit of $(F^0_*)$ also has dimension $\geq$ length $(F_*^0)$. So, 
\begin{equation}
\dim(N/(\Stab(F_*)\cap N))\geq\mbox{length}(F^0_*).
\end{equation}
Let $L(i):=$length$(F_{i*})$, pick indices $k_1<\dots<k_{L(i)}$ so that $F_{i*}$ is 
$$
0\subsetneq F_{i,k_{1}}\subsetneq\dots\subsetneq F_{i,k_{L(i)}}\subsetneq{E_{i+1}\over E_i},
$$ 
and let $V_{ij}:={F_{i,k_{j+1}}\over F_{i,k_j}}$, so that the associated graded of $F_{i*}$ is
$$
Gr(F_{i*})= V_{i0}\oplus V_{i1}\oplus\dots V_{iL(i)}.
$$ 
Now, the stabilizer of $F_{i*}$ splits as 
$$
\Stab(F_{i*})=N_i\rtimes\left(\prod_{j=0}^{length(F_{i*})} \GL(V_{ij})\right),
$$
where $N_i$ is the nilpotent group that acts trivially on the associated graded $Gr(F_{i*})$. 
Let $M^1_{ij}<\GL(V_{ij})$ be the subgroup of matrices of determinant $\pm 1$ and denote
$$
G:=(\Stab(F_*)\cap N)\rtimes\prod^{length(E_*)}_{i=0}\left(N_i\rtimes\left(\prod_{j=0}^{length(F_{i*})}M^1_{ij}\right)\right).
$$
This is a Lie subgroup of $\GL(\mathbb R^m)\cong N\rtimes(\prod_{i}(N_i\rtimes(\prod_j\GL(V_{ij}))))$. For every non-zero vector space $V_{ij}$, each $M^1_{ij}$ has codimension one in $\GL(V_{ij})$ and also $M^1_{ij}\cap \SO(m)$ has codimension one in $\GL(V_{ij})\cap \SO(m)$. The group $K=\prod_{ij}M^1_{ij}\cap\SO(m)=G\cap\SO(m)$ is a maximal compact subgroup of $G$ So, the codimension of $G/K$ in $\GL(\mathbb R^m)/\SO(m)$ is 
\begin{eqnarray*}
\label{codimension}
\dim(N/(\Stab(F_*)\cap N))&+&\sum_{i=0}^{length(E_*)}(1+length(F_{i*}))=\\
\dim(N/(\Stab(F_*)\cap N))&+&\left(\sum_{i=0}^{length(E_*)}length(F_{i*})\right)+length(E_*)+1\\
\geq length(F^0_*)&+&\left(\sum_{i=0}^{length(E_*)}length(F_{i*})\right)+length(E_*)+1\\
=length(F_*)&+&length(E_*)+1\\
=(|\sigma|+1)&+&(|\tau|+1)+1.
\end{eqnarray*} 
We can rewrite this conclusion as
\begin{eqnarray}
\label{slmbound}
\dim(G/K)+|\sigma|+|\tau|&\leq&\dim(\GL(\mathbb R^m)/\SO(m))-3\\
&=&\dim(\SL(\mathbb R^m)/\SO(m))-2\\
&<&\dim\partial,
\end{eqnarray}
where $\partial$ is the Borel-Serre boundary. 

Next we will prove the following claim. \subsection*{Claim} For any torsionfree subgroup $\Gamma<\SL(\mathbb Z^m)$ we have 
\begin{equation}
\label{hdim}
\mbox{hdim}(\Stab_{\Gamma}(E_*)\cap\Stab_{\Gamma}(F_*))\leq\dim(G/K).
\end{equation}
Together with (\ref{slmbound}) this claim implies that the complex of flags is a small singular model for the $\Gamma$-cover $\widetilde\partial\ra\partial$ of the Borel-Serre boundary. 
\subsection*{Proof of claim}
We will need the following simple lemma. 
\begin{lemma}
Suppose $\gamma\in \GL(\mathbb Q^m)$ is an element whose characteristic polynomial $p_{\gamma}(t):=\det(t-\gamma)$ has integer coefficients and $\det(\gamma)=\pm 1$. If $V\subset \mathbb Q^m$ is a $\gamma$-invariant subspace 
then $\det(\gamma\mid_V)=\pm 1$ and $\det(\gamma\mid_{(\mathbb Q^m/V)})=\pm 1$.  
\end{lemma} 
\begin{proof}
First, since $p_{\gamma}(t)$ is a monic polynomial with integer coefficients, any root $\lambda$ of $p_{\gamma}(t)$ is an algebraic integer. Second, $p_{\gamma}(0)=\pm\det(\gamma)=\pm 1$. Therefore $p_{\gamma}(t)=t\cdot q(t)\pm 1$ for some monic polynomial with integer coefficients $q(t)$. The inverse $\lambda^{-1}$ can be expressed as $\lambda^{-1}=\pm q(\lambda)$, so it is also an algebraic integer. 

Note that $\gamma$ acts on the associated graded vector space $V\oplus(\mathbb Q^{m}/V)\cong\mathbb Q^{m}$ with the same characteristic polynomial, so $p_{\gamma}$ factors as 
\begin{equation}
\label{factor}
p_{\gamma}(t)=\det(t-\gamma\mid_{V})\det(t-\gamma\mid_{(\mathbb Q^m/V)}).
\end{equation} 
Let $\lambda_1,\dots,\lambda_{\dim(V)}$ be the roots of the polynomial $\det(t-\gamma\mid_V)$. Since the $\lambda_i$ are also roots of $p_{\gamma}(t)$ both 
$
\det(\gamma\mid_V)=\pm\lambda_1\cdot\dots\cdot\lambda_{\dim(V)}
$
and its inverse $\det(\gamma\mid_V)^{-1}$ are algebraic integers. Since $\det(\gamma\mid_V)\in\mathbb Q$ is a rational number this implies both $\det(\gamma\mid_V)$ and $\det(\gamma\mid_V)^{-1}$ are integers, which can only happen if $\det(\gamma\mid_V)=\pm 1$. By (\ref{factor}) we also get $\det(\gamma\mid_{(\mathbb Q^m/V)})=\pm 1$.  
\end{proof}
Applying this lemma repeatedly, we conclude that the group 
$$
S:=\Stab_{\Gamma}(E_*)\cap\Stab_{\Gamma}(F_*)
$$ 
acts on $V_{ij}$ by matrices of determinant $\pm 1$. Therefore $S<G$.
Since $S$ is a torsionfree, discrete subgroup of $G$, it acts by covering translations on the quotient $G/K$. Since $G/K$ is contractible\footnote{It is built from symmetric manifolds by taking products and fibre bundles with simply connected nilpotent fibres.} we conclude that hdim$(S)\leq\dim(G/K).$ This finishes the proof of the claim and thus also finishes the proof of Proposition \ref{flagcomplexsmall}.


\section{Simply connected fillings (Finishing the proof of Theorem \ref{incompressible})}
\begin{proposition}
\label{simplyconnected}
Suppose $\widetilde\partial\sim\vee S^{q-1}$. If $\partial$ has a compressible, simply connected filling then $H_*(\partial)$ is torsionfree and $H_k(\partial)=0$ for $k\notin\{0,q-1,n-q,n-1\}$.\footnote{Here the dimension of $\partial$ is $n-1$.} 
\end{proposition}
\begin{proof}
Suppose $c:M\ra\partial$ is a compression of a simply connected, compressible filling $(M,\partial)$.  
\begin{itemize}
\item
Since $M$ is simply connected, the compression $c$ lifts to a homotopy section $\widetilde c:M\ra\widetilde\partial$ of $\widetilde\partial\ra M$. So, since $\widetilde\partial$ is homotopy equivalent to a wedge of $(q-1)$-spheres, $M$ is also homotopy equivalent to a (possibly smaller) wedge of $(q-1)$-spheres. In particular, the reduced homology $\overline H_*(M)$ is torsionfree and concentrated in dimension $q-1$.
\item
Using the compression $c:M\ra\partial$ and Poincare duality we get
$$
H_*(\partial)\cong H_*(M)\oplus H_{*+1}(M,\partial)=H_*(M)\oplus H^{n-1-*}(M).
$$
So, since $\overline H_*(M)$ is torsionfree and concentrated in dimension $q-1$, $H_*(\partial)$ is also torsionfree and concentrated in dimensions $\{0,q-1,n-q,n-1\}$. 
\end{itemize}
\end{proof}
\begin{example}
$H_1(BF_2^2)\not=0$ so we conclude from\footnote{For $d>2$, the proposition applies directly with $q=d$. For $d=2$ the boundary $\partial$ is aspherical so we can apply the proposition with $q=1$.} Proposition \ref{simplyconnected} that for $d\geq 2$ every simply connected filling of $\partial((\mathbb T^2_1)^d)$ is incompressible.  
\end{example}
\begin{example}
For any $m\geq 4$ there is a finite index torsionfree subgroup $\Gamma<\SL(\mathbb Z^m)$ whose abelianization $H_1(B\Gamma)$ is nontrivial\footnote{It follows from the normal subgroup theorem (and can also be proved directly) that $H_1(B\Gamma)$ is finite.}. Since the corresponding Borel-Serre boundary $\partial$ has fundamental group $\Gamma$, $H_1(\partial)=H_1(B\Gamma)$ and we conclude from Proposition \ref{simplyconnected} that every simply connected filling of $\partial$ is incompressible. 
\end{example}
\begin{example}
The Borel-Serre boundaries $\partial$ corresponding to finite index torsion free subgroups $\Gamma<\SL(\mathbb Z^m)$ have vanishing Euler characteristic $\chi(\partial)=0$. If $m=4k$ then both $q-1=m-2$ and $d={m\choose 2}$ are even, and $\partial$ is a manifold of dimension $d+q-1$. So, $\chi(\partial)=0$ implies that $\partial$ has some odd dimensional homology. Since $q-1,d,$ and $d+q-1$ are even, Proposition \ref{simplyconnected} implies that every simply connected filling of $\partial$ is incompressible. 
\end{example}
\subsection{\label{finish}Finishing the proof of Theorem \ref{incompressible}}
Suppose $m\geq4$, $\Gamma<\SL(\mathbb Z^m)$ is a finite index torsionfree subgroup and $(M,\partial)$ is a compressible filling of the Borel-Serre boundary $\partial$. Then the  fundamental group of the boundary $\pi_1\partial=\Gamma$ splits as a semidirect product $\Gamma\cong K\rtimes\pi_1M$. 
Theorem \ref{nonperipheral} and Proposition \ref{flagcomplexsmall} imply that the filling is not $\pi_1$-injective, so $K\not=0$. The Margulis normal subgroup theorem implies that $K$ is a finite index subgroup of $\Gamma$, so $\pi_1M$ is finite. Since $\Gamma$ is torsionfree, $\pi_1M$ must be trivial. Thus, any compressible filling of $\partial$ is simply connected. The above example shows if $m=4k$ there are no compressible, simply connected fillings of $\partial$, so we conclude that $\partial$ has no compressible fillings at all. This proves Theorem \ref{incompressible}.

\section*{Notation} For the rest of the paper, we assume $(\widetilde M,\widetilde\partial)\sim(*,\vee S^{q-1})$. Let $\Gamma=\pi_1M$ be the fundamental group and let $D=\overline H_{q-1}(\vee S^{q-1})$ be the reduced homology $\Gamma$-module. We also assume that $\Gamma$ has trivial center. 

\section{\label{smith}No homotopically trivial $\mathbb Z/p$-actions (proof of Theorem \ref{htrivial})}

In this section we use $\mathbb F_p$-coefficients. 
The proof is via the method in \cite{avramidi2011periodic}. We refer the reader to that paper for more details.

The bundle $\widetilde\partial\ra\widetilde\partial\times_{\Gamma}\widetilde\partial\ra\partial$  has a diagonal section $\partial\stackrel{s}\ra\widetilde{\partial}\times_{\Gamma}\widetilde{\partial}$ so its homology splits 
as 
$H_*(\widetilde{\partial}\times_{\Gamma}\widetilde{\partial})\cong H_*(\partial)\oplus H_{*-(q-1)}(\partial;D)$. Lemma \ref{equivariant} implies the diagonal section $s$ is a homology isomorphism in dimensions $\geq n-1$ and consequently 
\begin{equation}
\label{vanishequation}
H_{\geq n-q}(\partial;D)=H_{\geq (n-1)-(q-1)}(\partial;D)=0.
\end{equation}

It follows from the spectral sequence corresponding to $\widetilde\partial\ra\widetilde\partial\times_{\Gamma}\widetilde M\ra M$ that 
\begin{eqnarray}
\label{fund}
H_{n-q}(M;D)&\cong&\mathbb F_p,\\
\label{zero}
H_{>n-q}(M;D)&=&0.
\end{eqnarray}
Note also that the long exact sequence, (\ref{vanishequation}) and (\ref{zero}) imply 
\begin{equation}
\label{relativevanish}
H_{>n-q}(M,\partial;D)=0.
\end{equation}


Set $d=n-q$.  
Let $H^{cl}$ be homology with closed supports and $H^e$ the homology of the end (see \cite{avramidi2011periodic} for more on these notions). These are generalizations of relative homology and homology of the boundary, respectively, and defined in such a way that
\begin{eqnarray}
\label{closed}
H_*^{cl}(\dot{M};D)&\cong&H_*(M;\partial;D),\\
\label{end}
H_*^e(\dot{M};D)&\cong&H_*(\partial;D).
\end{eqnarray}
But, they also make sense if we do not know that a manifold is the interior of a compact manifold with boundary. In particular, they make sense for the fixed point set $F$ of a homotopically trivial $\mathbb Z/p$-action on the interior $\dot{M}$. The assumption that $\Gamma$ has trivial center implies that any homotopically trivial $\mathbb Z/p$-action on $M$ lifts to a $\mathbb Z/p$-action on $\widetilde M$ that commutes with $\Gamma$. Consequently (Theorem 7 in \cite{avramidi2011periodic}), we get Smith theory isomorphisms
\begin{eqnarray}
H_*(F;D)&\cong& H_*(\dot M;D),\\
H_*^{cl}(F;D)&\cong&H^{cl}_{*+(n-\dim(F))}(\dot M;D).
\end{eqnarray}

Now, look at the commutative diagram of long exact sequences
\begin{equation}
\begin{array}{ccccc}
0&&\mathbb F_p&&\\
||&&||&&\\
H_d^e(\dot{M};D)&\ra& H_d(\dot{M};D)\\
\uparrow&&||&&\\
H_d^e(F;D)&\stackrel{\phi}\ra&H_d(F;D)&\ra&H_d^{cl}(F;D)\\
&&&&||\\
&&&&H^{cl}_{d+(n-\dim(F))}(\dot M;D),
\end{array}
\end{equation} 
where the vertical isomorphisms come from Smith theory, the top left term is zero by (\ref{vanishequation}) and (\ref{end}), while the top middle term is $\mathbb F_p$ by (\ref{fund}). If $\dim(F)<n$ then, by (\ref{relativevanish}) and (\ref{closed}) the bottom right term is zero, implying $\phi$ is onto, which is a contradiction. So, we must have $\dim(F)=n$ i.e. the fixed set is the entire manifold $\dot{M}$. 
\section{Isometries of complete Riemannian metrics (Proof of Corollary \ref{isometries})}
If $g$ is a complete, Riemannian metric on the interior $\dot{M}$ then $\Isom(\dot{M},g)$ is a Lie group and the group of homotopically trivial isometries $K:=\ker(\Isom(\dot{M},g)\ra\Out(\Gamma))$ is, as well. If $\widetilde\partial$ has a small singular model, then Theorem \ref{nonperipheral} (in the cases $q=0$ and $1$ we use the variant given in subsection \ref{variant}) implies that the interior $\dot{M}$ contains a compact subset $Z$ that cannot be moved off itself by any homotopically trivial isometry $\psi\in K$.\footnote{That is, $\psi(Z)\cap Z\not=\emptyset$.} It follows from this that the group $K$ of all such isometries is compact. Theorem \ref{htrivial} implies $K$ has no finite order elements, so $K$ is the trivial group. 

\section{\label{mcgcorollaries}Applications to moduli spaces}
\begin{corollary}[Minimal orbifold theorem]
The moduli space $\mathcal M_g^{\pm}$ is a minimal orbifold. 
\end{corollary}
\begin{proof}
If $\mathcal M_g^{\pm}\ra Q$ is a finite-sheeted orbifold cover then we can find a manifold $\dot{M}\ra\mathcal M_g^{\pm}$ and a finite
group of diffeomorphisms $G$ of $\dot{M}$ such that $Q=\dot{M}/G$. By Proposition \ref{curvecomplexsmall} and Theorem \ref{htrivial}, $\dot{M}$ has no homotopically trivial $\mathbb Z/p$-actions,
so the map $G\ra\Out(\pi_1\dot{M})$ is injective. Ivanov's computation of commensurators of mapping class groups \cite{ivanov}
shows that $|\Out(\pi_1\dot{M})|=\deg(\dot{M}\ra\mathcal M_g^{\pm})$ which implies $\mathcal M_g^{\pm}=Q$.  
\end{proof}
\begin{corollary}[Topological analogue of Royden's theorem]
If $h$ is a complete finite volume Riemannian (or Finsler) metric on moduli space $\mathcal M_g^{\pm}$, 
then the isometry group $I$ of the lifted metric $\widetilde h$ on the universal cover is the extended mapping class group.
\end{corollary}
\begin{proof}
Since $\widetilde h$ is lifted from moduli space, the extended mapping class group acts by isometries of $\widetilde h$. 
Farb and Weinberger show in Theorem 1.2 of \cite{farbweinbergerroyden} 
that the isometry group $I$ contains the extended mapping class group as a finite index subgroup so we get an orbifold cover $\mathcal M_g^{\pm}\ra\dot{\widetilde M}/I$.
By the minimal orbifold theorem this cover is trivial, i.e. the isometry group is the extended mapping class group.
\end{proof}
\begin{remark}
In \cite{avramidi2011isometries} these two corollaries were proved using $L^2$-betti numbers instead of small singular models. 
\end{remark}

\section{\label{obstruction} The obstruction $\alpha_M$ (Description of $\alpha_M$ and proof of Proposition \ref{obstructionproposition})}
There are no obstructions to building a $\Gamma$-map $\widetilde f:\widetilde M^{(q-1)}\ra\widetilde\partial$ that is the identity on the boundary. Such a map defines a $\Gamma$-equivariant cochain $o_f:C_{q}(\widetilde M)\ra\overline H_{q-1}(\widetilde\partial)=D$ by sending each cell $c$ to the homology class of $\widetilde f(\partial c)$. This cochain is zero on the boundary, and it is actually a cocycle, 
so it gives an element $[o_f]^*\in H^q(M,\partial;D)$ in the relative cohomology group. It is the obstruction to extending $\widetilde f$ to a $\Gamma$-map $\widetilde g:\widetilde M^{(q)}\ra\widetilde\partial$ that is still the identity on the boundary, after possibly changing $\widetilde f$ on the $(q-1)$-cells to another $\Gamma$-map $\widetilde f':\widetilde M^{(q-1)}\ra\widetilde\partial$. 
Such an extension $\widetilde g$ does not exist (it would give a homotopy section $g$ of $\partial^{(q)}\hookrightarrow M^{(q)}$, and would imply that $\overline H_{q-1}(\widetilde\partial)=0$) so the cohomology class $[o_f]^*$ is non-zero. Moreover, this class does not depend on the choice of initial map $\widetilde f$, and so we simply call it $\alpha_M^*$. Its image in the absolute cohomology $H^q(M;D)$ is denoted $\alpha_M$. It is the obstruction to extending $\widetilde f$ to $\widetilde M^{(q)}$ after possibly changing it on some (interior or boundary) $(q-1)$-chains.      

\begin{proof}[Proof of Proposition \ref{obstructionproposition}]
Poincare duality and (\ref{vanishequation}) imply that \begin{equation}
H^{q-1}(\partial;D)\cong H_{n-q}(\partial;D)=0.
\end{equation} 
Since $\alpha_M^*\not=0$ the long exact homology sequence
\begin{equation}
\begin{array}{ccccc}
H^{q-1}(\partial;D)&\ra& H^q(M,\partial;D)&\ra& H^q(M;D),\\
||&&&&\\
0&&\alpha_M^*&\mapsto&\alpha_M
\end{array}
\end{equation}
implies $\alpha_M$ is also nonzero. 
\end{proof}
\subsection*{Relation to the fundamental class}
The cap product $\cap[M]$ gives a Poincare duality isomorphism $H^{q}(M,\partial;D)\cong H_{n-q}(M;D)\cong\mathbb Z$. Since $\alpha_M^*$ is a nonzero element on the left hand side, its Poincare dual $\alpha_M^*\cap[M]$ is a non-zero constant multiple of a generator $e$ of $H_{n-q}(M;D)$.

\section{Duality groups}
In more algebraic terms, $(\widetilde M,\widetilde\partial)\sim(*,\vee S^{q-1})$ says $\Gamma$ is an $(n-q)$-dimensional duality group of finite type\footnote{The group $\Gamma$ has {\it finite type} if it has a finite classifying space $B\Gamma$. It is a {\it $d$-dimensional duality group} if $H^*(B\Gamma;\mathbb Z\Gamma)$ is torsionfree and has homology concentrated in a single dimesion $d$.}. Let $d=n-q$. The coefficient module $D$ is identified with the dualizing module
$$
D=\overline H_{q-1}(\widetilde\partial)\cong H_{q}(\widetilde M,\widetilde\partial)\cong H_c^d(\widetilde M)\cong H^d(\widetilde M;\mathbb Z\Gamma)\cong H^d(B\Gamma;\mathbb Z\Gamma). 
$$
We also note that if $\widetilde\partial$ has a small singular model then Poincare duality and (\ref{relativevanish}) imply
\begin{equation}
H^{<q}(M;D)\cong H_{>n-q}(M,\partial;D)=0.
\end{equation} 
If $(M',\partial')\sim(*,\vee S^{q'-1})$ is another pair with $\pi_1M'\cong\Gamma$ then it has an obstruction $\alpha_{M'}\in H^{q'}(M';D).$ If $q=q'$, then we can compare $\alpha_M$ with $\alpha_{M'}$ by identifying both of the cohomology groups with $H^q(B\Gamma;D)$.  So, one wants to understand the structure of $H^*(B\Gamma;D)$.  

\begin{example}
If $\Gamma=F_2^d$ then the dualizing module $D$ is a tensor prodcut $D=D_1\otimes\dots\otimes D_d$, where the $D_i$ are the dualizing modules of the $F_2$-factors in $F_2^d$. From this it is not hard to see that $H^*(BF_2^d;D)$ is concentrated in dimension $d$ and infinitely generated as an abelian group.
\end{example}
This suggests the following
\begin{problem}
Compute the group $H^*(B\Gamma;D)\cong H_{d-*}(B\Gamma;D\otimes D)$ (is it concentrated in a single dimension? infinitely generated?) and determine which elements are obstructions $\alpha_M$. 
\end{problem}

\section{Obstructing compressible fillings via multiplication in cohomology}

In this section we illustrate how multiplication in the cohomology of $\partial$ sometimes obstructs existence of compressible fillings.
Recall that a compressible filling $M$ of $\partial$ gives a splitting \begin{equation}
\label{splitting}
H^*(\partial)\cong H^*(M)\oplus H^{*+1}(M,\partial)\cong H^*(M)\oplus H_{n-*}(M).
\end{equation}

\begin{proposition}
Let $m>1$. If $\partial$ is a closed, orientable $km$-manifold, $H^k(\partial;\mathbb Q)=\left<\omega\right>$ is one dimensional and $\omega^m$ in $H^{mk}(\partial;\mathbb Q)$ is nonzero, then $\partial$ has no compressible fillings. 
\end{proposition}
\begin{proof}
Suppose $i:\partial\hookrightarrow M$ is a filling with compression $c:M\ra \partial$. If $c^*\omega\not=0$ then $i^*$ and $c^*$ give inverse isomorphisms $H^k(\partial;\mathbb Q)\cong H^k(M;\mathbb Q)$, because $H^k(\partial;\mathbb Q)$ is one-dimensional. So $i^*(c^*\omega^m)=(i^*c^*\omega)^m=\omega^m\not=0$. Consequently $c^*\omega^m$ is a nonzero element in $H^{km}(M)$. But equation (\ref{splitting}) implies $H^{km}(M)=0$, giving a contradiction. Thus we must have $c^*\omega=0$ and $c^*\omega^{m-1}=0$.  But this implies that both $H^k(M;\mathbb Q)$ and $H^{(m-1)k}(M;\mathbb Q)$ vanish. The last one also implies that $H_{(m-1)k}(M;\mathbb Q)=0$. Consequently $H^k(\partial;\mathbb Q)=H^k(M;\mathbb Q)\oplus H_{(m-1)k}(M;\mathbb Q)=0$, again giving a contradiction.
\end{proof}
This applies, for instance, when $\partial$ is a complex projective space $\mathbb CP^n$. More generally, it applies to any closed K\"ahler (or symplectic) manifold $\partial$ with $b_2(\partial)=1$. Smooth complex hypersurfaces in complex projective space $\mathbb CP^n$ for $n\geq 4$ have this property and also generic intersections of such hypersurfaces, as long as the resulting manifold has complex dimension $\geq 3$ (by the Lefschetz hyperplane theorem \cite{bott}). Some of these manifolds are boundaries. For instance, the $\mathbb CP^{2k+1}$ bound: they are total spaces of $S^2$-bundles over quaternionic projective spaces $\mathbb HP^k$, so they bound the corresponding $D^3$-bundles. Thus, the $\mathbb CP^{odd}$ are boundaries with no compressible fillings. (The $\mathbb CP^{2k}$ do not bound. One way to see this is to note that they have nonzero signature.)

Any closed orientable $6$-dimensional manifold is a boundary (see \cite{milnorstasheff}, where one can also find lots of information on when higher dimensional manifolds bound.) On the other hand, it seems hard to find compressible fillings of K\"ahler manifolds even if $b_2>1$, except when the K\"ahler manifold is a product with a surface $\times \Sigma^2$ or, possibly, a surface bundle with section. This suggest the following
\begin{question}
If a closed, simply connected, $6$-dimensional K\"ahler manifold $\partial$ has a compressible filling, does it split as $N^4\times S^2$?
\end{question}
\begin{example}
Suppose $\partial$ satisfies the hypotheses of the question and, in addition $b_2(\partial)=2$. One can show that on the level of rational homotopy, existence of a compressible filling is equivalent to existence of a surjection of graded algebras 
$$
\phi:H^*(\partial;\mathbb Q)\twoheadrightarrow\mathbb Q[\beta]/\beta^3,
$$ 
where $\beta$ is an element of degree $2$. In particular, this only depends on $H^{even}(\partial;\mathbb Q)$ and not on $H^3$.
Let $\omega$ be the K\"ahler form, normalized so that $\int\omega^3=1$. Then, the existence of $\phi$ boils down to two conditions. 
\begin{itemize}
\item
There is an element $\beta\in H^2(\partial;\mathbb Q)$ so that $\beta^2\not=0$ and $\beta^3=0$, and 
\item
if we write $\omega\beta=x\omega^2+y\beta^2$ then\footnote{Multiplying the equation by $\omega$ and by $\beta$ one computes that $x={\int\omega\beta^2\over\int\omega^2\beta}$ and $y={\int\omega^2\beta\over\int\omega\beta^2}-{\int\omega^3\over\int\omega^2\beta}$.} the rational number $s$ given by $\phi(\omega)=s\beta$ must satisfy $s=xs^2+y$, so $(2xs-1)^2=1-4xy$. In particular, $1-4xy$ must be a rational square.   
\end{itemize}
In some cases, (e.g. generic intersections of hypersurfaces in $\mathbb CP^n\times\mathbb CP^m$) one can show that there isn't even a {\it real} surjection $H^*(\partial;\mathbb R)\ra\mathbb R[\beta]/\beta^3$, by showing that $1-4xy$ is negative.
\end{example}

\bibliography{locsym}
\bibliographystyle{plain}
\end{document}